\documentclass[english,onecolumn]{article}
\usepackage{lmodern}
\usepackage[T1]{fontenc}
\usepackage[latin9]{inputenc}
\usepackage{geometry}
\usepackage{amsmath}
\usepackage{amsthm}
\usepackage{amssymb}
\usepackage{graphicx}
\usepackage{esint}

\makeatletter
\numberwithin{equation}{section}
\numberwithin{figure}{section}
\theoremstyle{plain}
\newtheorem{thm}{\protect\theoremname}
\theoremstyle{definition}
\newtheorem{defn}[thm]{\protect\definitionname}
\theoremstyle{plain}
\newtheorem{cor}[thm]{\protect\corollaryname}
\theoremstyle{plain}
\newtheorem{conjecture}[thm]{\protect\conjecturename}

\usepackage{mathrsfs}
\usepackage{algorithm,algpseudocode}

\usepackage{colortbl}
\definecolor{lightgray}{rgb}{0.9,0.9,0.9}
\definecolor{lightred}{rgb}{1,0.8,0.8}
\definecolor{lightgreen}{rgb}{0.6,1,0.6}
\definecolor{lightyellow}{rgb}{1,1,0.5}
\definecolor{lightgrey}{rgb}{0.8,0.8,0.8}

\allowdisplaybreaks[1]

\makeatother

\usepackage{babel}
\providecommand{\conjecturename}{Conjecture}
\providecommand{\corollaryname}{Corollary}
\providecommand{\definitionname}{Definition}
\providecommand{\theoremname}{Theorem}

\begin{document}
\title{Pairwise Near-maximal Grand Coupling of Brownian Motions}
\author{Cheuk Ting Li and Venkat Anantharam\\
EECS, UC Berkeley, Berkeley, CA, USA\\
Email: ctli@berkeley.edu, ananth@eecs.berkeley.edu}
\maketitle
\begin{abstract}
The well-known reflection coupling gives a maximal coupling of two
one-dimensional Brownian motions with different starting points. Nevertheless,
the reflection coupling does not generalize to more than two Brownian
motions. In this paper, we construct a coupling of all Brownian motions
with all possible starting points (i.e., a grand coupling), such that
the coupling for any pair of the coupled processes is close to being
maximal, that is, the distribution of the coupling time of the pair
approaches that of the maximal coupling as the time tends to $0$
or $\infty$, and the coupling time of the pair is always within a
multiplicative factor $2e^{2}$ from the maximal one. We also show
that a grand coupling that is pairwise exactly maximal does not exist.
\end{abstract}

\section{Introduction}

The maximal coupling of two stochastic processes $P,Q$ is a coupling
$(\{X_{t}\}_{t},\{Y_{t}\}_{t})$ (i.e., the marginal distribution
of $\{X_{t}\}_{t}$ is $P$, and that of $\{Y_{t}\}_{t}$ is $Q$)
that simultaneously maximizes the probability that the processes match
after time $s$ (i.e., $X_{t}=Y_{t}$ for all $t\ge s$) for all $s$.
It was studied by Griffeath \cite{griffeath1975maximal}, Goldstein
\cite{goldstein1979maximal} and Pitman \cite{pitman1976coupling}.
For two one-dimensional Brownian motions with different starting points,
a maximal coupling can be given by the reflection coupling studied
by Lindvall \cite{lindvall1982coupling}, Lindvall and Rogers \cite{lindvall1986coupling},
Hsu and Sturm \cite{hsu2013maximal}, and Kendall \cite{kendall2015coupling}.
Also see \cite{arous1995coupling,kendall2004coupling,kendall2007coupling}
for results on coupling functionals of Brownian motions.

While the maximal coupling of two stochastic processes exists under
rather general conditions \cite{goldstein1979maximal,thorisson1986maximal},
it might not exist in the pairwise sense for more than two processes,
that is, given a collection of stochastic processes $\{P_{\alpha}\}_{\alpha\in\mathcal{A}}$,
there may not exist a coupling $\{\{X_{\alpha,t}\}_{t}\}_{\alpha\in\mathcal{A}}$
(i.e., the marginal distribution of $\{X_{\alpha,t}\}_{t}$ is $P_{\alpha}$)
that simultaneously maximizes $\mathbf{P}(\forall t\ge s:\,X_{\alpha,t}=X_{\beta,t})$
for all $s$, $\alpha$, $\beta$. The maximal coalescent coupling,
which maximizes the probability that all processes in the collection
match after time $s$ (i.e., $\mathbf{P}(\forall\alpha,\beta\in\mathcal{A},\,t\ge s:\,X_{\alpha,t}=X_{\beta,t})$),
was studied by Connor \cite{connor2007coupling}. Nevertheless, a
maximal coalescent coupling, which only concerns whether the processes
all agree after certain time, may not give a maximal (or close to
maximal) coupling when only the marginal distribution of a pair of
processes $\{X_{\alpha,t}\}_{t}$, $\{X_{\beta,t}\}_{t}$ is considered
(refer to Section \ref{sec:dyadic}). Other related works on the coupling
of more than two distributions or stochastic processes include coupling
from the past \cite{propp1996exact,propp1998coupling}, Wasserstein
barycenter \cite{agueh2011barycenters}, and multi-marginal optimal
transport \cite{kellerer1984duality,gangbo1998optimal,pass2011uniqueness,angel2019pairwise,li2019pairwise}.
A coupling of Markov chains with the same Markov kernel and all possible
initial states (i.e., $P_{\alpha}$ is the Markov chain starting at
$\alpha$ for any state $\alpha\in\mathcal{A}$) is often called a
grand coupling in the literature on coupling from the past and mixing
times of Markov chains (e.g. \cite{levin2010glauber}).

A classical example of a grand coupling of all one-dimensional Brownian
motions with all possible starting points (i.e., $P_{\alpha}=\mathrm{BM}(\alpha)$,
the Brownian motion starting at $\alpha\in\mathbb{R}$) is the Brownian
web studied by Arratia \cite{arratia1980coalescing} and T{\'o}th
and Werner \cite{toth1998true}.\footnote{The Brownian web is a coupling of all Brownian motions starting at
every two-dimensional point in space-time. In this paper, we only
consider starting points at time $0$.} The Brownian web has a property that, if we consider the marginal
joint distribution of the processes with distribution $\mathrm{BM}(\alpha)$
and $\mathrm{BM}(\beta)$ ($\alpha\neq\beta\in\mathbb{R}$), then
the processes move independently from $\alpha$ and $\beta$ respectively,
until they couple (become equal), and then move together (the same
as the Doeblin coupling for Markov chains \cite{doeblin1938expose},
which is generally not maximal). The distribution of the coupling
time between the two processes is the same as the distribution of
twice the coupling time of the reflection coupling, i.e., the Brownian
web has a multiplicative gap $2$ from the optimum (refer to Section
\ref{sec:dyadic}). The multiplicative gap does not vanish as the
time tends to $0$ or $\infty$.

In this paper, we give a grand coupling $\{\{X_{\alpha,t}\}_{t}\}_{\alpha\in\mathbb{R}}$
of all one-dimensional Brownian motions with all possible starting
points (i.e., $\{X_{\alpha,t}\}_{t}$ has marginal $P_{\alpha}=\mathrm{BM}(\alpha)$
for $\alpha\in\mathbb{R}$), called the \emph{dyadic grand coupling},
such that the coupling for any pair of the coupled processes is close
to being maximal. Let
\[
\Upsilon_{\alpha,\beta}:=\inf\big\{ s\ge0:\,X_{\gamma_{1},t}=X_{\gamma_{2},t},\,\forall\,\gamma_{1},\gamma_{2}\in[\alpha,\beta],t\ge s\big\}
\]
be the coupling time of all the Brownian motions with starting point
lying in the interval $[\alpha,\beta]$, and $\hat{\Upsilon}_{\alpha,\beta}:=\inf\{s\ge0:\,\hat{X}_{\alpha,t}=\hat{X}_{\beta,t},\,\forall\,t\ge s\}$
be the coupling time of the reflection coupling, where $\{\hat{X}_{\alpha,t}\}_{t\ge0},\{\hat{X}_{\beta,t}\}_{t\ge0}$
is the reflection coupling of $\mathrm{BM}(\alpha),\mathrm{BM}(\beta)$
(which is a maximal coupling). Let ``$\preceq$'' denote first-order
stochastic dominance between two real-valued random variables (i.e.,
$Y\preceq Z$ if $\mathbf{P}(Y\ge t)\le\mathbf{P}(Z\ge t)$ for all
$t\in\mathbb{R}$). Then the distribution of $\Upsilon_{\alpha,\beta}$
is close to that of the optimal $\hat{\Upsilon}_{\alpha,\beta}$ for
all $\alpha<\beta$, in the sense that $\Upsilon_{\alpha,\beta}\preceq2e^{2}\hat{\Upsilon}_{\alpha,\beta}$,
and the distribution of $\Upsilon_{\alpha,\beta}$ tends to that of
$\hat{\Upsilon}_{\alpha,\beta}$ as the time tends to $0$ or $\infty$
(in the sense of multiplicative gap). More precisely, there exists
a function $r:\mathbb{R}_{>0}\to[1,2e^{2}]$ (that does not depend
on $\alpha,\beta$) such that $\lim_{t\to0}r(t)=\lim_{t\to\infty}r(t)=1$,
and
\begin{equation}
\hat{\Upsilon}_{\alpha,\beta}\preceq\Upsilon_{\alpha,\beta}\preceq r\left(\frac{\hat{\Upsilon}_{\alpha,\beta}}{|\alpha-\beta|^{2}}\right)\hat{\Upsilon}_{\alpha,\beta}\label{eq:stoc_dom}
\end{equation}
for any $\alpha<\beta$. Numerical evidence shows that the maximum
multiplicative gap $2e^{2}$ can be improved to around $1.5$, and
the dyadic grand coupling has a strictly smaller coupling time than
the Brownian web in the sense of first-order stochastic dominance
(see Figure \ref{fig:bdr}). Refer to Section \ref{sec:dyadic} for
details.

A natural question is whether there exists a grand coupling $\{\{\tilde{X}_{\alpha,t}\}_{t}\}_{\alpha\in\mathbb{R}}$
of $\{\mathrm{BM}(\alpha)\}_{\alpha\in\mathbb{R}}$ such that any
pair $\{\tilde{X}_{\alpha,t}\}_{t},\{\tilde{X}_{\beta,t}\}_{t}$ is
a maximal coupling. In Section \ref{sec:nonexist}, we show that such
a pairwise maximal grand coupling does not exist. We conjecture that
the dyadic grand coupling is optimal, in the sense of attainable failure
probability bounds, as defined in Section \ref{sec:nonexist}.

\medskip{}

\[
\]

\section{Dyadic Grand Coupling of Brownian Motion\label{sec:dyadic}}

Let $\mathrm{BM}(\alpha)$ be the distribution of the standard one-dimensional
Brownian motion starting at $\alpha\in\mathbb{R}$ ($P_{\alpha}$
is a probability distribution over the space of continuous functions
$\mathrm{C}([0,\infty),\mathbb{R})$ with the topology of uniform
convergence over compact subsets of $[0,\infty)$). To couple $\mathrm{BM}(\alpha),\mathrm{BM}(\beta)$
with two different starting points $\alpha,\beta$, the reflection
coupling \cite{lindvall1982coupling,lindvall1986coupling} $\{\hat{X}_{\alpha,t}\}_{t\ge0},\{\hat{X}_{\beta,t}\}_{t\ge0}$
is given by $\{\hat{X}_{\alpha,t}\}_{t\ge0}\sim\mathrm{BM}(\alpha)$,
$T:=\inf\{t\ge0:\,\hat{X}_{\alpha,t}=(\alpha+\beta)/2\}$, $\hat{X}_{\beta,t}=\alpha+\beta-\hat{X}_{\alpha,t}$
for $t<T$, $\hat{X}_{\beta,t}=\hat{X}_{\alpha,t}$ for $t\ge T$.
The probability of failure of the reflection coupling can be given
by
\begin{equation}
\mathbf{P}\Big(\exists\,t\ge s\,\mathrm{s.t.}\,\hat{X}_{\alpha,t}\neq\hat{X}_{\beta,t}\Big)=\mathrm{erf}\left(\frac{|\alpha-\beta|}{2\sqrt{2s}}\right)\label{eq:pfail_reflect}
\end{equation}
for any $s>0$, where 
\begin{align*}
\mathrm{erf}(\gamma) & :=\int_{-\gamma}^{\gamma}\frac{e^{-x^{2}}}{\sqrt{\pi}}\mathrm{d}x
\end{align*}
is the error function.

Nevertheless, if we have to couple all the processes in $\{\mathrm{BM}(\alpha)\}_{\alpha\in\mathbb{R}}$,
it is impossible to simultaneously attain this probability of failure
for all pairs of starting points, as will be shown in Section \ref{sec:nonexist}.
The maximal coalescent coupling \cite{connor2007coupling} is not
useful in this setting since, for any fixed time, it is impossible
for all the processes in $\{\{X_{\alpha,t}\}_{t\ge0}\}_{\alpha\in\mathbb{R}}$
(where $\{X_{\alpha,t}\}_{t\ge0}\sim\mathrm{BM}(\alpha)$) to coalesce
(become all equal) by that time with a positive probability. If we
consider only the processes $\{\mathrm{BM}(\alpha)\}_{\alpha\in[\gamma_{1},\gamma_{2}]}$
with starting points in the interval $[\gamma_{1},\gamma_{2}]$, then
a maximal coalescent coupling can be given simply by performing the
reflection coupling between $\{X_{\gamma_{1},t}\}_{t}$ and $\{X_{\alpha,t}\}_{t}$
for all $\alpha\in(\gamma_{1},\gamma_{2}]$ (note that in the reflection
coupling, one process can be obtained deterministically from another,
and thus we can express $\{X_{\alpha,t}\}_{t}$ as a function of $\{X_{\gamma_{1},t}\}_{t}\sim\mathrm{BM}(\gamma_{1})$
for all $\alpha\in(\gamma_{1},\gamma_{2}]$). This coupling is undesirable
since the coupling time between $\{X_{\gamma_{2},t}\}_{t}$ and $\{X_{\gamma_{2}-\epsilon,t}\}_{t}$
is the same as that between $\{X_{\gamma_{2},t}\}_{t}$ and $\{X_{\gamma_{1},t}\}_{t}$,
despite $\mathrm{BM}(\gamma_{2})$ being much closer to $\mathrm{BM}(\gamma_{2}-\epsilon)$
than to $\mathrm{BM}(\gamma_{1})$.

The Brownian web \cite{arratia1980coalescing,toth1998true} $\{\{X_{\alpha,t}^{\mathrm{BW}}\}_{t\ge0}\}_{\alpha\in\mathbb{R}}$
(where $\{X_{\alpha,t}^{\mathrm{BW}}\}_{t\ge0}\sim\mathrm{BM}(\alpha)$)
gives a probability of failure
\begin{align*}
 & \mathbf{P}\Big(\exists\,t\ge s\,\mathrm{s.t.}\,X_{\alpha,t}^{\mathrm{BW}}\neq X_{\beta,t}^{\mathrm{BW}}\Big)\\
 & =\mathbf{P}\Big(\exists\,\gamma_{1},\gamma_{2}\in[\alpha,\beta],t\ge s\,\mathrm{s.t.}\,X_{\gamma_{1},t}^{\mathrm{BW}}\neq X_{\gamma_{2},t}^{\mathrm{BW}}\Big)\\
 & =\mathrm{erf}\left(\frac{|\alpha-\beta|}{2\sqrt{s}}\right),
\end{align*}
and hence the distribution of the coupling time between $\{X_{\alpha,t}^{\mathrm{BW}}\}_{t}$
and $\{X_{\beta,t}^{\mathrm{BW}}\}_{t}$ (the first time where $X_{\alpha,t}^{\mathrm{BW}}=X_{\beta,t}^{\mathrm{BW}}$)
is the same as the distribution of twice the coupling time of the
reflection coupling $\{\hat{X}_{\alpha,t}\}_{t},\{\hat{X}_{\beta,t}\}_{t}$.
The multiplicative gap $2$ does not vanish as the time tends to $0$
or $\infty$, that is, the Brownian web does not satisfy \eqref{eq:stoc_dom}.

In this section, we propose a coupling that achieves a probability
of failure close to that of the reflection coupling for all pairs
of starting points. We construct a coupling $\{\{X_{\alpha,t}\}_{t\ge0}\}_{\alpha\in\mathbb{R}}$
(where $\{X_{\alpha,t}\}_{t\ge0}\sim\mathrm{BM}(\alpha)$) as follows.
\begin{defn}
[Dyadic grand coupling of Brownian motion]\label{def:dyadic}Let
$\{Y_{t}\}_{t\ge0}\sim\mathrm{BES}^{3}(0)$, the Bessel process of
dimension $3$ starting at $0$. Let $W_{i}\stackrel{iid}{\sim}\mathrm{Unif}\{\pm1\}$
for $i\in\mathbb{Z}$ be independent of $\{Y_{t}\}_{t\ge0}$. For
any $\theta\in[0,1]$ and $\alpha\in\mathbb{R}$, let
\[
G_{\theta,\alpha,j}=G_{\theta,\alpha,j}(\{W_{i}\}_{i\le j}):=\begin{cases}
W_{j} & \mathrm{if}\;\left(\alpha-\sum_{k=-\infty}^{j-1}W_{k}2^{k+\theta-1}\!+2^{j+\theta-1}\right)\mathrm{mod}\;2^{j+\theta+1}\in[0,2^{j+\theta}),\\
-W_{j} & \mathrm{otherwise},
\end{cases}
\]
for $j\in\mathbb{Z}$, where $a\;\mathrm{mod}\;b:=a-b\lfloor a/b\rfloor$
for $b>0$. By the definition of $G_{\theta,\alpha,j}$, the conditional
distribution of $G_{\theta,\alpha,j}$ given any $\{W_{k}\}_{k<j}$
is $\mathrm{Unif}\{\pm1\}$ (since $W_{j}\sim\mathrm{Unif}\{\pm1\}$
independent of $\{W_{k}\}_{k<j}$). Hence $G_{\theta,\alpha,j}\stackrel{iid}{\sim}\mathrm{Unif}\{\pm1\}$,
and is independent of $\{Y_{t}\}_{t\ge0}$.

As will be shown in Appendix \ref{sec:pf_dyadic}, if $\sup\{j:W_{j}=1\}=\sup\{j:W_{j}=-1\}=\infty$
(which happens almost surely), then
\begin{equation}
\left\lfloor 2^{-(j+\theta)}\left(\alpha-\!\sum_{k=-\infty}^{j-1}W_{k}2^{k+\theta-1}\right)+\frac{1}{2}\right\rfloor =\sum_{k=j}^{\infty}(W_{k}-G_{\theta,\alpha,k})2^{k-j-1}\label{eq:dyadic_partial}
\end{equation}
for any $j\in\mathbb{Z}$,
\begin{equation}
\sum_{j=-\infty}^{\infty}(W_{j}-G_{\theta,\alpha,j})2^{j+\theta-1}=\alpha,\label{eq:dyadic_sum}
\end{equation}
and $G_{\theta,\alpha,j}=W_{j}$ for all sufficiently large $j$.
For $\theta\in[0,1]$ and $i\in\mathbb{Z}$, let
\[
T_{\theta,i}:=\inf\left\{ t\ge0:\,Y_{t}=2^{i+\theta}\right\} .
\]
Let $X_{\theta,\alpha,0}:=\alpha$. For $t>0$, with $i$ defined
to satisfy $T_{\theta,i-1}<t\le T_{\theta,i}$, let
\begin{align}
X_{\theta,\alpha,t} & :=\alpha+\sum_{j=-\infty}^{i-1}G_{\theta,\alpha,j}2^{j+\theta-1}+(Y_{t}-2^{i+\theta-1})G_{\theta,\alpha,i}\label{eq:dyadic_xdef_f}\\
 & =(Y_{t}-2^{i+\theta})G_{\theta,\alpha,i}+\sum_{j=-\infty}^{\infty}(W_{j}-\mathbf{1}\{j>i\}G_{\theta,\alpha,j})2^{j+\theta-1},\label{eq:dyadic_xdef_b}
\end{align}
where the equivalence of \eqref{eq:dyadic_xdef_f} and \eqref{eq:dyadic_xdef_b}
can be seen by \eqref{eq:dyadic_sum}. Let $\Theta\sim\mathrm{Unif}[0,1]$
be independent of $(\{Y_{t}\}_{t\ge0},\{W_{i}\}_{i})$, and let $X_{\alpha,t}:=X_{\Theta,\alpha,t}$.
\end{defn}

We now check that $\{X_{\alpha,t}\}_{t\ge0}\sim\mathrm{BM}(\alpha)$.
First, for any $i$, we can see from \eqref{eq:dyadic_xdef_f} that
$\lim_{t\searrow T_{\theta,i-1}}X_{\theta,\alpha,t}=X_{\theta,\alpha,T_{\theta,i-1}}$,
and hence $\{X_{\alpha,t}\}_{t\ge0}$ is continuous in $t$.

By the strong Markov property, $\{Y_{t+T_{\theta,i-1}}\}_{t\ge0}\sim\mathrm{BES}^{3}(2^{i+\theta-1})$.
Since $\mathrm{BES}^{3}(2^{i+\theta-1})$ is the distribution of $\mathrm{BM}(2^{i+\theta-1})$
conditioned to stay positive \cite{mckean1963excursions,knight1969brownian,williams1974path,pitman1975one},
we see that $\{Y_{t+T_{\theta,i-1}}\}_{0\le t\le T_{\theta,i}-T_{\theta,i-1}}$
has the same distribution as $\mathrm{BM}(2^{i+\theta-1})$ conditioned
to stay positive and stopped when it hits $2^{i+\theta}$, or equivalently,
stopped when it hits either $0$ or $2^{i+\theta}$ and conditioned
on the event that it hits $2^{i+\theta}$ (which has probability $1/2$,
so the conditioning is well-defined). By symmetry, $\{(Y_{t+T_{\theta,i-1}}-2^{i+\theta-1})G_{\theta,\alpha,i}\}_{0\le t\le T_{\theta,i}-T_{\theta,i-1}}$
has the same distribution as $\mathrm{BM}(0)$ stopped when it hits
either $2^{i+\theta-1}$ or $-2^{i+\theta-1}$, and is independent
of $(T_{\theta,i-1},\{Y_{t}\}_{t\le T_{\theta,i-1}},\{G_{\theta,\alpha,j}\}_{j\le i-1})$
(since $G_{\theta,\alpha,j}\stackrel{iid}{\sim}\mathrm{Unif}\{\pm1\}$
independent of $\{Y_{t}\}_{t\ge0}$). Welding these processes together,
we can see from \eqref{eq:dyadic_xdef_f} that $\{X_{\theta,\alpha,t+T_{\theta,i}}-X_{\theta,\alpha,T_{\theta,i}}\}_{t\ge0}$
follows $\mathrm{BM}(0)$ and is independent of $(T_{\theta,i},\{Y_{t}\}_{t\le T_{\theta,i}},\{G_{\theta,\alpha,j}\}_{j\le i})$
for any $\theta,\alpha,i$. Since a random process with continuous
sample paths is characterized by its finite-dimensional marginals,
by letting $i\to-\infty$, we see that $\{X_{\theta,\alpha,t}\}_{t\ge0}\sim\mathrm{BM}(\alpha)$
for any $\theta,\alpha$. \footnote{More precisely, for any $0<\tau_{1}<\cdots<\tau_{m}$, we have $\{X_{\theta,\alpha,\tau_{k}+T_{\theta,i}}-X_{\theta,\alpha,T_{\theta,i}}\}_{k}\to\{X_{\theta,\alpha,\tau_{k}}-\alpha\}_{k}$
as $i\to-\infty$ almost surely since $T_{\theta,i}\to0$ and $X_{\theta,\alpha,T_{\theta,i}}\to\alpha$.
Since $\{X_{\theta,\alpha,\tau_{k}+T_{\theta,i}}-X_{\theta,\alpha,T_{\theta,i}}\}_{k}$
has the same distribution as $\{B_{\tau_{k}}\}_{k}$ (where $\{B_{t}\}_{t}$
is the Brownian motion), $\{X_{\theta,\alpha,\tau_{k}}-\alpha\}_{k}$
also has the same distribution as $\{B_{\tau_{k}}\}_{k}$.} Hence $\{X_{\alpha,t}\}_{t\ge0}\sim\mathrm{BM}(\alpha)$.

\begin{figure}
\begin{centering}
\includegraphics[scale=0.39]{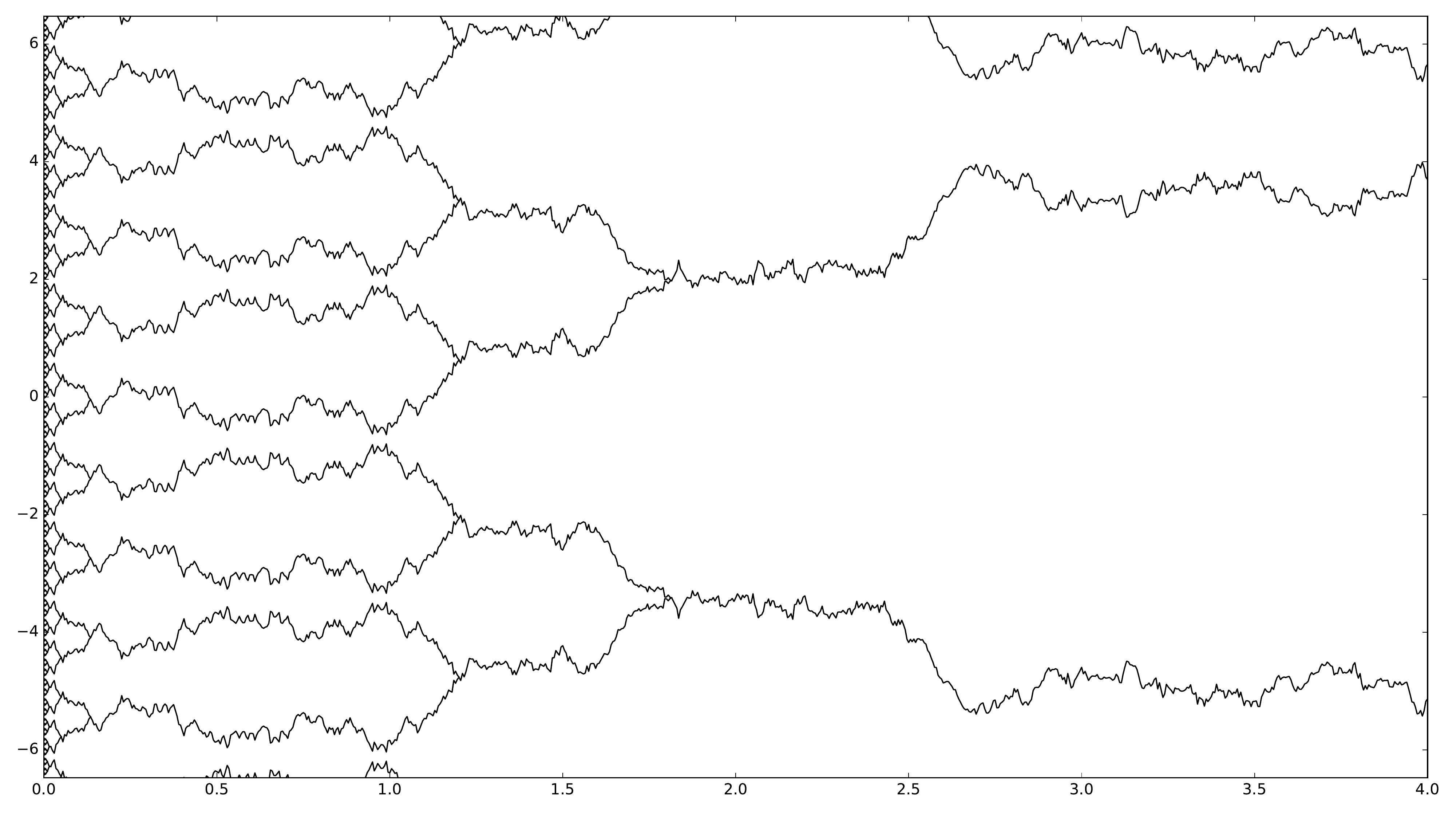}
\par\end{centering}
\caption{\label{fig:sample}A sample of the dyadic grand coupling, where all
the processes $\{\{X_{\alpha,t}\}_{t}\}_{\alpha\in\mathbb{R}}$ are
plotted together. Note that the coalescence points (the points where
two processes join) at the same time are evenly spaced on the space
axis. The processes after the time of each coalescence point can be
regarded as performing the reflection coupling between adjacent pairs
of coalescence points.}
\end{figure}

We then evaluate the probability of failure of this coupling.
\begin{thm}
\label{thm:b_dyadic}For the dyadic grand coupling of Brownian motion
$\{\{X_{\alpha,t}\}_{t\ge0}\}_{\alpha\in\mathbb{R}}$, we have
\begin{align*}
 & \mathbf{P}\Big(\exists\,t\ge s\,\mathrm{s.t.}\,X_{\alpha,t}\neq X_{\beta,t}\Big)\\
 & =\mathbf{P}\Big(\exists\,\gamma_{1},\gamma_{2}\in[\alpha,\beta],t\ge s\,\mathrm{s.t.}\,X_{\gamma_{1},t}\neq X_{\gamma_{2},t}\Big)\\
 & =\int_{\psi/2}^{\infty}\left(\sum_{k=1}^{\infty}2(-1)^{k+1}\exp\left(-\frac{k^{2}\pi^{2}}{2\zeta^{2}}\right)\right)\frac{\zeta^{-2}\psi}{\ln2}\left(\min\left\{ \zeta\psi^{-1},\,1\right\} -\frac{1}{2}\right)\mathrm{d}\zeta
\end{align*}
for any $\alpha<\beta$ and $s>0$, where $\psi:=|\alpha-\beta|/\sqrt{s}$.
\end{thm}

\medskip{}

As a consequence, we have the following results.
\begin{cor}
\label{cor:b_tail}Let $\{\{X_{\alpha,t}\}_{t\ge0}\}_{\alpha\in\mathbb{R}}$
be the one-dimensional dyadic grand coupling of Brownian motion. Fix
any $\alpha<\beta$. Let $\{\hat{X}_{\alpha,t}\}_{t\ge0},\{\hat{X}_{\beta,t}\}_{t\ge0}$
be the reflection coupling of $\mathrm{BM}(\alpha),\mathrm{BM}(\beta)$.
Let $\Upsilon_{\alpha,\beta}:=\inf\{s\ge0:\,X_{\gamma_{1},t}=X_{\gamma_{2},t},\,\forall\,\gamma_{1},\gamma_{2}\in[\alpha,\beta],t\ge s\}$
and $\hat{\Upsilon}_{\alpha,\beta}:=\inf\{s\ge0:\,\hat{X}_{\alpha,t}=\hat{X}_{\beta,t},\,\forall\,t\ge s\}$
be the coupling times of the dyadic grand coupling and the reflection
coupling respectively. Let $F_{\Upsilon_{\alpha,\beta}}^{-1}(p):=\inf\{s:\,\mathbf{P}(\Upsilon_{\alpha,\beta}\le s)\ge p\}$
be the inverse distribution function of $\Upsilon_{\alpha,\beta}$,
and define $F_{\hat{\Upsilon}_{\alpha,\beta}}^{-1}(p)$ similarly.
We have:
\begin{enumerate}
\item \label{enu:b_tail_tail}For any $s>0$ (let $\psi:=|\alpha-\beta|/\sqrt{s}$),
\[
\mathbf{P}(\Upsilon_{\alpha,\beta}>s)\le\frac{\psi}{\sqrt{2\pi}}.
\]
As a result,
\[
\lim_{s\to\infty}\frac{\mathbf{P}(\Upsilon_{\alpha,\beta}>s)}{\mathbf{P}(\hat{\Upsilon}_{\alpha,\beta}>s)}=1,
\]
and
\[
\lim_{p\to1}\frac{F_{\Upsilon_{\alpha,\beta}}^{-1}(p)}{F_{\hat{\Upsilon}_{\alpha,\beta}}^{-1}(p)}=1,
\]
i.e., the tail of the distribution of the coupling time of the dyadic
grand coupling approaches that of the reflection coupling as $s\to\infty$.
\item \label{enu:b_tail_head}For any $s>0$ (let $\psi:=|\alpha-\beta|/\sqrt{s}$),
if $\psi\ge2\sqrt{2}$, then
\[
\mathbf{P}(\Upsilon_{\alpha,\beta}>s)\le1-\left(1-\left(\mathrm{erf}\left(\frac{\psi+8/\psi}{2\sqrt{2}}\right)\right)^{3}\right)\frac{\ln(1+8/\psi^{2})+(1+8/\psi^{2})^{-1}-1}{\ln2}.
\]
As a result, 
\[
\lim_{p\to0}\frac{F_{\Upsilon_{\alpha,\beta}}^{-1}(p)}{F_{\hat{\Upsilon}_{\alpha,\beta}}^{-1}(p)}=1,
\]
i.e., the multiplicative gap between $\Upsilon_{\alpha,\beta}$ and
$\hat{\Upsilon}_{\alpha,\beta}$ vanishes as $s\to0$.
\item \label{enu:b_tail_unif}For any $s>0$ (let $\psi:=|\alpha-\beta|/\sqrt{s}$),
\[
\mathbf{P}(\Upsilon_{\alpha,\beta}>s)\le\mathrm{erf}\left(\frac{e\psi}{2}\right).
\]
As a result, $2e^{2}\hat{\Upsilon}_{\alpha,\beta}$ first-order stochastically
dominates $\Upsilon_{\alpha,\beta}$, i.e., the dyadic grand coupling
is pairwise within a multiplicative factor $2e^{2}$ from being maximal.
\end{enumerate}
\end{cor}

\medskip{}

These three bounds imply \eqref{eq:stoc_dom} by taking 
\[
r(t):=\frac{1}{t}F_{\Upsilon_{0,1}}^{-1}(F_{\hat{\Upsilon}_{0,1}}(t)).
\]
Note that $\Upsilon_{\alpha,\beta}/|\alpha-\beta|^{2}$ has the same
distribution as $\Upsilon_{0,1}$, and $\hat{\Upsilon}_{\alpha,\beta}/|\alpha-\beta|^{2}$
has the same distribution as $\hat{\Upsilon}_{0,1}$.

\begin{figure}
\begin{centering}
\includegraphics[scale=0.39]{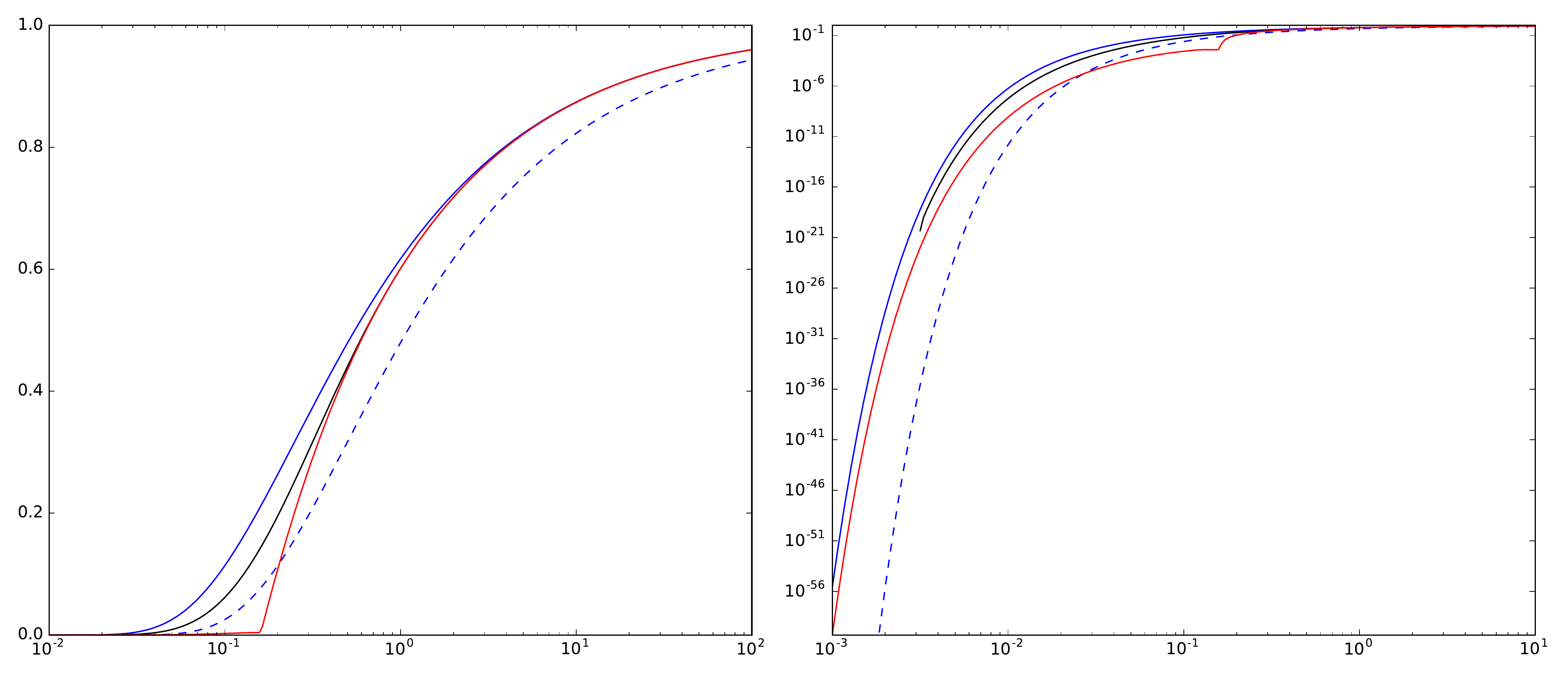}
\par\end{centering}
\caption{\label{fig:bd}Plot of the cumulative distribution function $F_{\Upsilon_{0,1}}$
(black), $F_{\hat{\Upsilon}_{0,1}}$ (blue), the bound on $F_{\Upsilon_{0,1}}$
in Corollary \ref{cor:b_tail} (red) (we take the pointwise maximum
of the three bounds in Corollary \ref{cor:b_tail}), and the cumulative
distribution function of the coupling time of the Brownian web (dashed
line). The left figure is in log-scale for the x-axis, whereas the
right figre is in log-scale for both axes. Note that $F_{\Upsilon_{0,1}}$
(the black curve) is bounded between the blue curve and the red curve.
Due to numerical precision issue, $F_{\Upsilon_{0,1}}(t)$ is not
plotted for small $t$'s in the right figure.}

\end{figure}

\begin{figure}
\begin{centering}
\includegraphics[scale=0.39]{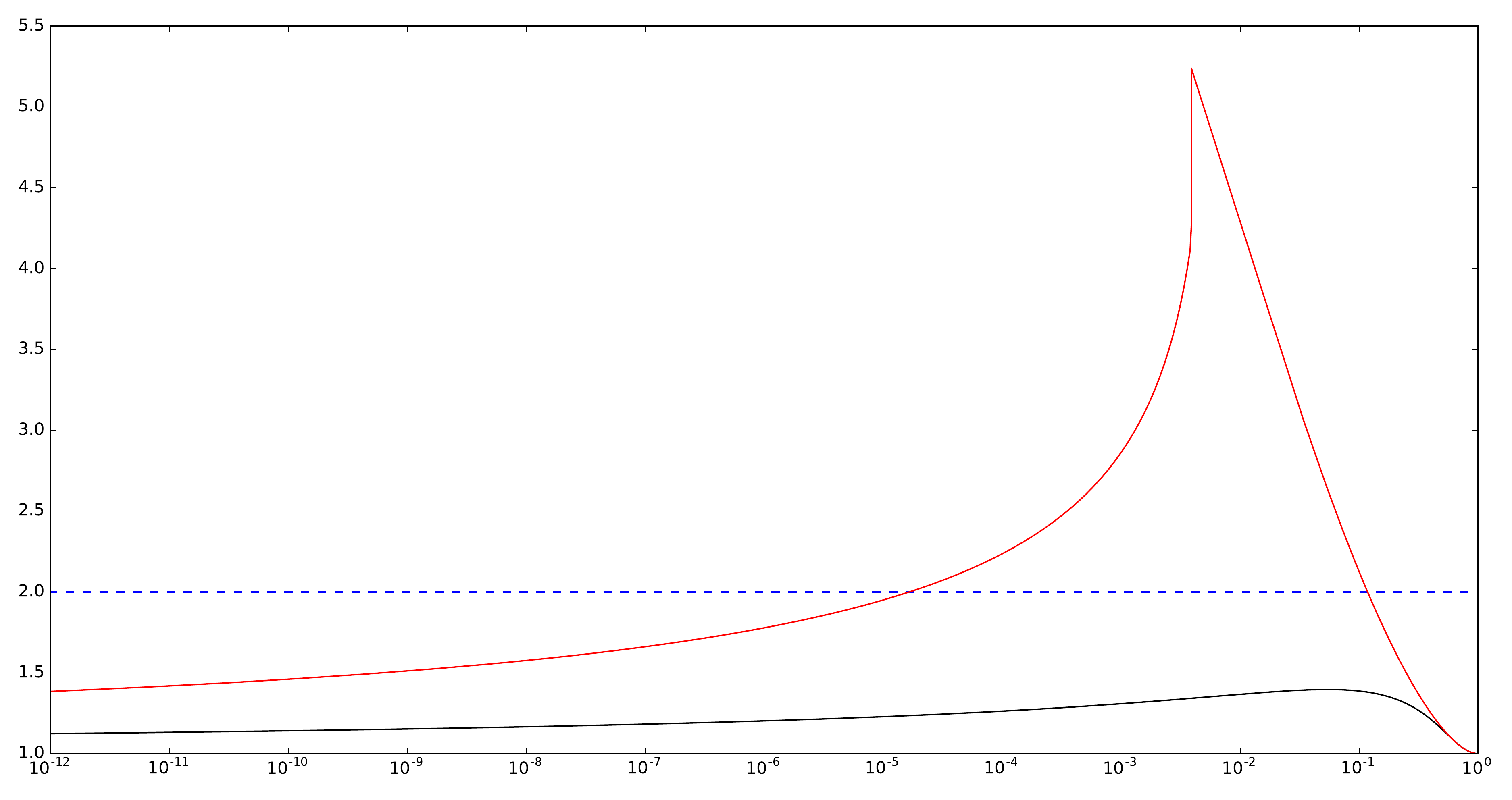}
\par\end{centering}
\caption{\label{fig:bdr}Plot of $F_{\Upsilon_{\alpha,\beta}}^{-1}(p)/F_{\hat{\Upsilon}_{\alpha,\beta}}^{-1}(p)$
(black), the upper bound on $F_{\Upsilon_{\alpha,\beta}}^{-1}(p)/F_{\hat{\Upsilon}_{\alpha,\beta}}^{-1}(p)$
in Corollary \ref{cor:b_tail} (red), and the corresponding ratio
for the Brownian web (dashed line, which is constantly $2$) against
$p$, where $\alpha<\beta$ (these curves do not depend on the choice
of $\alpha,\beta$). While Corollary \ref{cor:b_tail} gives a multiplicative
gap $2e^{2}$, we can observe in this graph that the multiplicative
gap can be improved to around $1.5$, since $F_{\Upsilon_{\alpha,\beta}}^{-1}(p)/F_{\hat{\Upsilon}_{\alpha,\beta}}^{-1}(p)$
stays below $1.5$ for all $p$.}
\end{figure}

We first prove Theorem \ref{thm:b_dyadic}.
\begin{proof}
[Proof of Theorem \ref{thm:b_dyadic}] Let $I$ be such that $T_{\Theta,I-1}<s\le T_{\Theta,I}$.
We have 
\begin{align}
 & \mathbf{P}\Big(\exists\,t\ge s\;\mathrm{s.t.}\,X_{\alpha,t}\neq X_{\beta,t}\Big)\nonumber \\
 & \stackrel{(a)}{=}\mathbf{P}\left(\exists\,k\ge I\;\mathrm{s.t.}\,G_{\Theta,\alpha,k}\neq G_{\Theta,\beta,k}\right)\nonumber \\
 & =\mathbf{P}\left(\exists\,k\ge I\;\mathrm{s.t.}\,\sum_{j=k}^{\infty}(W_{j}-G_{\Theta,\alpha,j})2^{j-k-1}\neq\sum_{j=k}^{\infty}(W_{j}-G_{\Theta,\beta,j})2^{j-k-1}\right)\nonumber \\
 & \stackrel{(b)}{=}\mathbf{P}\left(\exists\,k\ge I\;\mathrm{s.t.}\,\left\lfloor 2^{-(k+\Theta)}\!\left(\!\alpha\!-\!\sum_{j=-\infty}^{k-1}W_{j}2^{j+\Theta-1}\!\right)\!+\frac{1}{2}\right\rfloor \neq\left\lfloor 2^{-(k+\Theta)}\!\left(\!\beta\!-\!\sum_{j=-\infty}^{k-1}W_{j}2^{j+\Theta-1}\!\right)\!+\frac{1}{2}\right\rfloor \right)\nonumber \\
 & =\mathbf{P}\left(s\le T_{\Theta,\max\left\{ k:\,\left\lfloor 2^{-(k+\Theta)}\left(\alpha-\sum_{j=-\infty}^{k-1}W_{j}2^{j+\Theta-1}\right)+\frac{1}{2}\right\rfloor \neq\left\lfloor 2^{-(k+\Theta)}\left(\beta-\sum_{j=-\infty}^{k-1}W_{j}2^{j+\Theta-1}\right)+\frac{1}{2}\right\rfloor \right\} }\right)\nonumber \\
 & =\mathbf{P}\left(\sup_{t\le s}Y_{t}\le2^{\max\left\{ k:\,\left\lfloor 2^{-(k+\Theta)}\left(\alpha-\sum_{j=-\infty}^{k-1}W_{j}2^{j+\Theta-1}\right)+\frac{1}{2}\right\rfloor \neq\left\lfloor 2^{-(k+\Theta)}\left(\beta-\sum_{j=-\infty}^{k-1}W_{j}2^{j+\Theta-1}\right)+\frac{1}{2}\right\rfloor \right\} +\Theta}\right)\label{eq:p_sup_bd}\\
 & \stackrel{(c)}{=}\mathbf{E}\left[\sum_{k=1}^{\infty}2(-1)^{k+1}\exp\left(-\frac{k^{2}\pi^{2}}{2Z^{2}}\right)\right],\nonumber 
\end{align}
where 
\[
Z:=s^{-1/2}2^{\max\left\{ k:\,\left\lfloor 2^{-(k+\Theta)}\left(\alpha-\sum_{j=-\infty}^{k-1}W_{j}2^{j+\Theta-1}\right)+\frac{1}{2}\right\rfloor \neq\left\lfloor 2^{-(k+\Theta)}\left(\beta-\sum_{j=-\infty}^{k-1}W_{j}2^{j+\Theta-1}\right)+\frac{1}{2}\right\rfloor \right\} +\Theta}.
\]
Here (a) comes from \eqref{eq:dyadic_xdef_b}, (b) is due to \eqref{eq:dyadic_partial},
and (c) is because $(\pi/2)\sqrt{s}/\sup_{t\le s}Y_{t}$ follows the
Kolmogorov distribution \cite{borodin2012handbook,pitman1999law}.
By the same arguments,
\begin{align*}
 & \mathbf{P}\Big(\exists\,\gamma_{1},\gamma_{2}\in[\alpha,\beta],t\ge s\;\mathrm{s.t.}\,X_{\gamma_{1},t}\neq X_{\gamma_{2},t}\Big)\\
 & =\mathbf{P}\left(\exists\,\gamma_{1},\gamma_{2}\in[\alpha,\beta],k\ge I\;\mathrm{s.t.}\,G_{\Theta,\gamma_{1},k}\neq G_{\Theta,\gamma_{2},k}\right)\\
 & =\!\mathbf{P}\!\!\left(\!\exists\,\gamma_{1},\!\gamma_{2}\!\in\![\alpha,\beta],k\!\ge\!I\,\mathrm{s.t.}\left\lfloor 2^{-(k+\Theta)}\!\!\left(\!\gamma_{1}\!-\!\!\!\!\sum_{j=-\infty}^{k-1}\!\!\!W_{j}2^{j+\Theta-1}\!\right)\!\!+\!\frac{1}{2}\right\rfloor \!\neq\!\left\lfloor 2^{-(k+\Theta)}\!\!\left(\!\gamma_{2}\!-\!\!\!\!\sum_{j=-\infty}^{k-1}\!\!\!W_{j}2^{j+\Theta-1}\!\right)\!\!+\!\frac{1}{2}\right\rfloor \!\right)\\
 & =\mathbf{P}\left(\exists\,k\ge I\;\mathrm{s.t.}\,\left\lfloor 2^{-(k+\Theta)}\!\left(\!\alpha\!-\!\sum_{j=-\infty}^{k-1}W_{j}2^{j+\Theta-1}\!\right)\!+\frac{1}{2}\right\rfloor \neq\left\lfloor 2^{-(k+\Theta)}\!\left(\!\beta\!-\!\sum_{j=-\infty}^{k-1}W_{j}2^{j+\Theta-1}\!\right)\!+\frac{1}{2}\right\rfloor \right),
\end{align*}
and hence the two probabilities in Theorem \ref{thm:b_dyadic} are
equal.

To find the distribution of $Z$, we have
\begin{align*}
 & \mathbf{P}\left(\max\left\{ k:\,\left\lfloor 2^{-(k+\theta)}\left(\alpha-\sum_{j=-\infty}^{k-1}W_{j}2^{j+\theta-1}\right)+\frac{1}{2}\right\rfloor \neq\left\lfloor 2^{-(k+\theta)}\left(\beta-\sum_{j=-\infty}^{k-1}W_{j}2^{j+\theta-1}\right)+\frac{1}{2}\right\rfloor \right\} \ge l\right)\\
 & \stackrel{(a)}{=}\mathbf{P}\left(\exists\,k\ge l\,\mathrm{s.t.}\,\sum_{j=k}^{\infty}(W_{j}-G_{\theta,\alpha,j})2^{j-k-1}\neq\sum_{j=k}^{\infty}(W_{j}-G_{\theta,\beta,j})2^{j-k-1}\right)\\
 & =\mathbf{P}\left(\sum_{j=l}^{\infty}(W_{j}-G_{\theta,\alpha,j})2^{j-l-1}\neq\sum_{j=l}^{\infty}(W_{j}-G_{\theta,\beta,j})2^{j-l-1}\right)\\
 & \stackrel{(b)}{=}\mathbf{P}\left(\left\lfloor 2^{-(l+\theta)}\left(\alpha-\sum_{j=-\infty}^{l-1}W_{j}2^{j+\theta-1}\right)+\frac{1}{2}\right\rfloor \neq\left\lfloor 2^{-(l+\theta)}\left(\beta-\sum_{j=-\infty}^{l-1}W_{j}2^{j+\theta-1}\right)+\frac{1}{2}\right\rfloor \right)\\
 & \stackrel{(c)}{=}\frac{1}{2}\int_{-1}^{1}\mathbf{1}\left\{ \left\lfloor 2^{-(l+\theta)}\left(\alpha-\gamma2^{l+\theta-1}\right)+\frac{1}{2}\right\rfloor \neq\left\lfloor 2^{-(l+\theta)}\left(\beta-\gamma2^{l+\theta-1}\right)+\frac{1}{2}\right\rfloor \right\} \mathrm{d}\gamma\\
 & =\min\left\{ 2^{-(l+\theta)}|\alpha-\beta|,\,1\right\} ,
\end{align*}
where (a) and (b) are due to \eqref{eq:dyadic_partial}, and (c) is
because $\sum_{j=-\infty}^{l-1}W_{j}2^{j+\theta-1}\sim\mathrm{Unif}[-2^{l+\theta-1},2^{l+\theta-1}]$.
Therefore,
\begin{align*}
 & \mathbf{P}\left(Z\ge\zeta\right)\\
 & =\mathbf{P}\Bigg(\max\left\{ k:\,\left\lfloor 2^{-(k+\Theta)}\left(\alpha-\sum_{j=-\infty}^{k-1}W_{j}2^{j+\Theta-1}\right)+\frac{1}{2}\right\rfloor \neq\left\lfloor 2^{-(k+\Theta)}\left(\beta-\sum_{j=-\infty}^{k-1}W_{j}2^{j+\Theta-1}\right)+\frac{1}{2}\right\rfloor \right\} \\
 & \;\;\;\;\;\;\ge\left\lceil \log_{2}(\zeta\sqrt{s})-\Theta\right\rceil \Bigg)\\
 & =\int_{0}^{1}\min\left\{ 2^{-(\lceil\log_{2}(\zeta\sqrt{s})-\theta\rceil+\theta)}|\alpha-\beta|,\,1\right\} \mathrm{d}\theta\\
 & =\int_{0}^{1}\min\left\{ 2^{-(\log_{2}(\zeta\sqrt{s})+\theta)}|\alpha-\beta|,\,1\right\} \mathrm{d}\theta\\
 & =\int_{0}^{1}\min\left\{ \zeta^{-1}s^{-1/2}2^{-\theta}|\alpha-\beta|,\,1\right\} \mathrm{d}\theta\\
 & =\int_{0}^{1}\min\left\{ \zeta^{-1}2^{-\theta}\psi,\,1\right\} \mathrm{d}\theta,
\end{align*}
where $\psi:=|\alpha-\beta|/\sqrt{s}$. Hence,
\begin{align*}
 & -\frac{\mathrm{d}}{\mathrm{d}\zeta}\mathbf{P}\left(Z\ge\zeta\right)\\
 & =-\int_{0}^{1}\frac{\mathrm{d}}{\mathrm{d}\zeta}\min\left\{ \zeta^{-1}2^{-\theta}\psi,\,1\right\} \mathrm{d}\theta\\
 & =\int_{0}^{1}\mathbf{1}\left\{ \zeta^{-1}2^{-\theta}\psi\le1\right\} \zeta^{-2}2^{-\theta}\psi\mathrm{d}\theta\\
 & =\zeta^{-2}\psi\int_{\min\{\max\{\log_{2}(\zeta^{-1}\psi),\,0\},\,1\}}^{1}2^{-\theta}\mathrm{d}\theta\\
 & =\frac{\zeta^{-2}\psi}{\ln2}\left(2^{-\min\{\max\{\log_{2}(\zeta^{-1}\psi),\,0\},\,1\}}-\frac{1}{2}\right),
\end{align*}
and thus
\begin{align*}
 & \mathbf{P}\Big(\exists\,\gamma_{1},\gamma_{2}\in[\alpha,\beta],t\ge s\,\mathrm{s.t.}\,X_{\gamma_{1},t}\neq X_{\gamma_{2},t}\Big)\\
 & =\int_{0}^{\infty}\left(\sum_{k=1}^{\infty}2(-1)^{k+1}\exp\left(-\frac{k^{2}\pi^{2}}{2\zeta^{2}}\right)\right)\frac{\zeta^{-2}\psi}{\ln2}\left(2^{-\min\{\max\{\log_{2}(\zeta^{-1}\psi),\,0\},\,1\}}-\frac{1}{2}\right)\mathrm{d}\zeta\\
 & =\int_{\psi/2}^{\infty}\left(\sum_{k=1}^{\infty}2(-1)^{k+1}\exp\left(-\frac{k^{2}\pi^{2}}{2\zeta^{2}}\right)\right)\frac{\zeta^{-2}\psi}{\ln2}\left(\min\left\{ \zeta\psi^{-1},\,1\right\} -\frac{1}{2}\right)\mathrm{d}\zeta.
\end{align*}
\end{proof}
\medskip{}

We then prove Corollary \ref{cor:b_tail}.
\begin{proof}
[Proof of Corollary \ref{cor:b_tail}] We first prove Corollary
\ref{cor:b_tail}.\ref{enu:b_tail_tail}. By Theorem \ref{thm:b_dyadic},
\begin{align*}
 & \mathbf{P}\Big(\exists\,\gamma_{1},\gamma_{2}\in[\alpha,\beta],t\ge s\,\mathrm{s.t.}\,X_{\gamma_{1},t}\neq X_{\gamma_{2},t}\Big)\\
 & =\int_{\psi/2}^{\infty}\left(\sum_{k=1}^{\infty}2(-1)^{k+1}\exp\left(-\frac{k^{2}\pi^{2}}{2\zeta^{2}}\right)\right)\frac{\zeta^{-2}\psi}{\ln2}\left(\min\left\{ \zeta\psi^{-1},\,1\right\} -\frac{1}{2}\right)\mathrm{d}\zeta\\
 & =\sum_{k=1}^{\infty}2(-1)^{k+1}\int_{\psi/2}^{\infty}\exp\left(-\frac{k^{2}\pi^{2}}{2\zeta^{2}}\right)\frac{\zeta^{-2}\psi}{\ln2}\left(\min\left\{ \zeta\psi^{-1},\,1\right\} -\frac{1}{2}\right)\mathrm{d}\zeta\\
 & =\sum_{k=1}^{\infty}2(-1)^{k+1}\left(\frac{\frac{\psi}{\sqrt{2\pi}k}\mathrm{erf}\left(\frac{\pi k}{\sqrt{2}\psi}\right)-\mathrm{Ei}\left(-\frac{\pi^{2}k^{2}}{2\psi^{2}}\right)}{2\ln2}-\frac{\frac{\psi}{\sqrt{2\pi}k}\mathrm{erf}\left(\frac{\sqrt{2}\pi k}{\psi}\right)-\mathrm{Ei}\left(-\frac{2\pi^{2}k^{2}}{\psi^{2}}\right)}{2\ln2}+\frac{\frac{\psi}{\sqrt{2\pi}k}\mathrm{erf}\left(\frac{\pi k}{\sqrt{2}\psi}\right)}{2\ln2}\right)\\
 & =\frac{1}{\ln2}\sum_{k=1}^{\infty}(-1)^{k+1}\left(\frac{2\psi}{\sqrt{2\pi}k}\mathrm{erf}\left(\frac{\pi k}{\sqrt{2}\psi}\right)-\frac{\psi}{\sqrt{2\pi}k}\mathrm{erf}\left(\frac{\sqrt{2}\pi k}{\psi}\right)-\mathrm{Ei}\left(-\frac{\pi^{2}k^{2}}{2\psi^{2}}\right)+\mathrm{Ei}\left(-\frac{2\pi^{2}k^{2}}{\psi^{2}}\right)\right),
\end{align*}
where $\mathrm{Ei}(\gamma):=-\int_{-\gamma}^{\infty}(e^{-x}/x)\mathrm{d}x$
is the exponential integral function. We have
\begin{align}
 & \left.\frac{\mathrm{d}}{\mathrm{d}\psi}\mathbf{P}\Big(\exists\,\gamma_{1},\gamma_{2}\in[\alpha,\beta],t\ge s\,\mathrm{s.t.}\,X_{\gamma_{1},t}\neq X_{\gamma_{2},t}\Big)\right|_{\psi=0}\nonumber \\
 & =\frac{1}{\ln2}\sum_{k=1}^{\infty}(-1)^{k+1}\left.\frac{\mathrm{d}}{\mathrm{d}\psi}\left(\frac{2\psi}{\sqrt{2\pi}k}\mathrm{erf}\left(\frac{\pi k}{\sqrt{2}\psi}\right)-\frac{\psi}{\sqrt{2\pi}k}\mathrm{erf}\left(\frac{\sqrt{2}\pi k}{\psi}\right)-\mathrm{Ei}\left(-\frac{\pi^{2}k^{2}}{2\psi^{2}}\right)+\mathrm{Ei}\left(-\frac{2\pi^{2}k^{2}}{\psi^{2}}\right)\right)\right|_{\psi=0}\nonumber \\
 & =\frac{1}{\ln2}\sum_{k=1}^{\infty}(-1)^{k+1}\left(\frac{2}{\sqrt{2\pi}k}-\frac{1}{\sqrt{2\pi}k}\right)\nonumber \\
 & =\frac{1}{\sqrt{2\pi}\ln2}\sum_{k=1}^{\infty}\frac{(-1)^{k+1}}{k}\nonumber \\
 & =\frac{1}{\sqrt{2\pi}}.\label{eq:p_psi_deriv}
\end{align}
Also, for any $\alpha_{1}<\alpha_{2}<\alpha_{3}$,
\begin{align*}
 & \mathbf{P}\Big(\exists\,\gamma_{1},\gamma_{2}\in[\alpha_{1},\alpha_{3}],t\ge s\,\mathrm{s.t.}\,X_{\gamma_{1},t}\neq X_{\gamma_{2},t}\Big)\\
 & \le\mathbf{P}\Big(\exists\,\gamma_{1},\gamma_{2}\in[\alpha_{1},\alpha_{2}],t\ge s\,\mathrm{s.t.}\,X_{\gamma_{1},t}\neq X_{\gamma_{2},t}\Big)+\mathbf{P}\Big(\exists\,\gamma_{1},\gamma_{2}\in[\alpha_{2},\alpha_{3}],t\ge s\,\mathrm{s.t.}\,X_{\gamma_{1},t}\neq X_{\gamma_{2},t}\Big).
\end{align*}
Hence $\mathbf{P}(\exists\,\gamma_{1},\gamma_{2}\in[\alpha,\beta],t\ge s\,\mathrm{s.t.}\,X_{\gamma_{1},t}\neq X_{\gamma_{2},t})$
(which only depends on $\psi$) is subadditive in $\psi$ (in fact,
it is shown in Appendix \ref{sec:pf_concave} that it is concave).
Combining this with \eqref{eq:p_psi_deriv}, we have
\[
\mathbf{P}\Big(\exists\,\gamma_{1},\gamma_{2}\in[\alpha,\beta],t\ge s\,\mathrm{s.t.}\,X_{\gamma_{1},t}\neq X_{\gamma_{2},t}\Big)\le\frac{\psi}{\sqrt{2\pi}}.
\]

For Corollary \ref{cor:b_tail}.\ref{enu:b_tail_unif}, if $\psi\ge2$,
by \eqref{eq:p_sup_bd},
\begin{align*}
 & \mathbf{P}\Big(\exists\,\gamma_{1},\gamma_{2}\in[\alpha,\beta],t\ge s\,\mathrm{s.t.}\,X_{\gamma_{1},t}\neq X_{\gamma_{2},t}\Big)\\
 & =\mathbf{P}\big(\sup_{t\le s}Y_{t}\le\sqrt{s}Z\big)\\
 & =\mathbf{P}\big(\sup_{t\le1}Y_{t}\le Z\big)\\
 & \stackrel{(a)}{\le}\mathbf{P}\big(\sup_{t\le1}B_{1,t}\le Z\;\mathrm{and}\;\sup_{t\le1}B_{2,t}\le Z\;\mathrm{and}\;\sup_{t\le1}B_{3,t}\le Z\big)\\
 & =\mathbf{E}\left[\left(\mathrm{erf}\left(Z/\sqrt{2}\right)\right)^{3}\right]\\
 & \le\mathbf{E}\left[\mathrm{erf}\left(\exp(\ln Z)/\sqrt{2}\right)\right]\\
 & \stackrel{(b)}{\le}\mathrm{erf}\left(\exp(\mathbf{E}\left[\ln Z\right])/\sqrt{2}\right),
\end{align*}
where in (a), we let $Y_{t}=\sqrt{B_{1,t}^{2}+B_{2,t}^{2}+B_{3,t}^{2}}$,
where $\{B_{1,t}\}_{t},\{B_{2,t}\}_{t},\{B_{3,t}\}_{t}$ are independent
Brownian motions, and (b) is because $\gamma\mapsto\mathrm{erf}(\exp(\gamma)/\sqrt{2})$
is concave for $\gamma\ge0$ (note that $Z\ge\psi/2\ge1$). We have
\begin{align*}
 & \mathbf{E}\left[\ln Z\right]\\
 & =\int_{\psi/2}^{\infty}(\ln\zeta)\frac{\zeta^{-2}\psi}{\ln2}\left(\min\left\{ \zeta\psi^{-1},\,1\right\} -\frac{1}{2}\right)\mathrm{d}\zeta\\
 & =\frac{\psi\ln\psi+\psi+\psi(\ln\psi)^{2}}{2\psi\ln2}-\frac{\psi\ln(\psi/2)+\psi+(\psi/2)(\ln(\psi/2))^{2}}{2(\psi/2)\ln2}+\frac{\psi}{2\ln2}\cdot\frac{\ln\psi+1}{\psi}\\
 & =\frac{\ln\psi+1+(\ln\psi)^{2}}{2\ln2}-\frac{2\ln\psi-2\ln2+2+(\ln\psi-\ln2)^{2}}{2\ln2}+\frac{\ln\psi+1}{2\ln2}\\
 & =\frac{1}{2\ln2}\left((\ln\psi)^{2}-(\ln\psi-\ln2)^{2}\right)+1\\
 & =\frac{1}{2\ln2}(2\ln\psi-\ln2)\ln2+1\\
 & =\ln\psi-\frac{\ln2}{2}+1.
\end{align*}
Therefore,
\begin{align*}
 & \mathbf{P}\Big(\exists\,\gamma_{1},\gamma_{2}\in[\alpha,\beta],t\ge s\,\mathrm{s.t.}\,X_{\gamma_{1},t}\neq X_{\gamma_{2},t}\Big)\\
 & \le\mathrm{erf}\left(\frac{1}{\sqrt{2}}\exp\left(\ln\psi-\frac{\ln2}{2}+1\right)\right)\\
 & =\mathrm{erf}\left(\frac{e\psi}{2}\right).
\end{align*}
If $\psi<2$, then
\begin{align*}
 & \mathbf{P}\Big(\exists\,\gamma_{1},\gamma_{2}\in[\alpha,\beta],t\ge s\,\mathrm{s.t.}\,X_{\gamma_{1},t}\neq X_{\gamma_{2},t}\Big)\\
 & \le\frac{\psi}{\sqrt{2\pi}}\\
 & \le\mathrm{erf}\left(\frac{e\psi}{2}\right).
\end{align*}
The result follows.

For Corollary \ref{cor:b_tail}.\ref{enu:b_tail_head}, for any $1<\delta\le2$,
\begin{align*}
 & \mathbf{P}\Big(\exists\,\gamma_{1},\gamma_{2}\in[\alpha,\beta],t\ge s\,\mathrm{s.t.}\,X_{\gamma_{1},t}\neq X_{\gamma_{2},t}\Big)\\
 & \le\mathbf{E}\left[\left(\mathrm{erf}\left(Z/\sqrt{2}\right)\right)^{3}\right]\\
 & \le1-\left(1-\left(\mathrm{erf}\left(\frac{\psi\delta}{2\sqrt{2}}\right)\right)^{3}\right)\mathbf{P}(Z\le\psi\delta/2)\\
 & =1-\left(1-\left(\mathrm{erf}\left(\frac{\psi\delta}{2\sqrt{2}}\right)\right)^{3}\right)\int_{\psi/2}^{\psi\delta/2}\frac{\zeta^{-2}\psi}{\ln2}\left(\min\left\{ \zeta\psi^{-1},\,1\right\} -\frac{1}{2}\right)\mathrm{d}\zeta\\
 & =1-\left(1-\left(\mathrm{erf}\left(\frac{\psi\delta}{2\sqrt{2}}\right)\right)^{3}\right)\left(\frac{\psi}{2(\ln2)(\psi\delta/2)}+\frac{\ln(\psi\delta/2)}{\ln2}-\frac{\psi}{2(\ln2)(\psi/2)}-\frac{\ln(\psi/2)}{\ln2}\right)\\
 & =1-\left(1-\left(\mathrm{erf}\left(\frac{\psi\delta}{2\sqrt{2}}\right)\right)^{3}\right)\frac{\ln\delta+\delta^{-1}-1}{\ln2}
\end{align*}
If $\psi\ge2\sqrt{2}$, substituting $\delta=1+8/\psi^{2}$, we have
\begin{align*}
 & \mathbf{P}\Big(\exists\,\gamma_{1},\gamma_{2}\in[\alpha,\beta],t\ge s\,\mathrm{s.t.}\,X_{\gamma_{1},t}\neq X_{\gamma_{2},t}\Big)\\
 & \le1-\left(1-\left(\mathrm{erf}\left(\frac{\psi+8/\psi}{2\sqrt{2}}\right)\right)^{3}\right)\frac{\ln(1+8/\psi^{2})+(1+8/\psi^{2})^{-1}-1}{\ln2}\\
 & \le1-\left(1-\mathrm{erf}\left(\frac{\psi+8/\psi}{2\sqrt{2}}\right)\right)\frac{\ln(1+8/\psi^{2})+(1+8/\psi^{2})^{-1}-1}{\ln2}.
\end{align*}
Note that $\ln(1+8/\psi^{2})+(1+8/\psi^{2})^{-1}-1=\Omega(\psi^{-4})$
as $\psi\to\infty$, whereas $1-\mathrm{erf}(x)\to0$ exponentially
as $x\to\infty$. Therefore there exists a function $\kappa:\mathbb{R}_{>0}\to\mathbb{R}_{>0}$
such that $\lim_{t\to\infty}\kappa(t)=0$, and
\begin{align*}
 & \mathbf{P}\Big(\exists\,\gamma_{1},\gamma_{2}\in[\alpha,\beta],t\ge s\,\mathrm{s.t.}\,X_{\gamma_{1},t}\neq X_{\gamma_{2},t}\Big)\\
 & \le\mathrm{erf}\left(\frac{(1+\kappa(t))\psi}{2\sqrt{2}}\right).
\end{align*}
This implies that
\[
\lim_{p\to0}\frac{F_{\Upsilon_{\alpha,\beta}}^{-1}(p)}{F_{\hat{\Upsilon}_{\alpha,\beta}}^{-1}(p)}=1.
\]

\end{proof}
\smallskip{}

\section{Nonexistence of a Pairwise Maximal Coupling\label{sec:nonexist}}

In this section, we show that there does not exist a pairwise maximal
coupling $\{\{\tilde{X}_{\alpha,t}\}_{t\ge0}\}_{\alpha\in\mathbb{R}}$
of $\{\mathrm{BM}(\alpha)\}_{\alpha\in\mathbb{R}}$, that is, one
in which every pair $\{\tilde{X}_{\alpha,t}\}_{t\ge0},\{\tilde{X}_{\beta,t}\}_{t\ge0}$
is a maximal coupling, i.e.,
\begin{equation}
\mathbf{P}\Big(\exists\,t\ge s\,\mathrm{s.t.}\,\tilde{X}_{\alpha,t}\neq\tilde{X}_{\beta,t}\Big)=\mathrm{erf}\left(\frac{|\alpha-\beta|}{2\sqrt{2s}}\right).\label{eq:pfail_pairmax}
\end{equation}

Note that both \eqref{eq:pfail_pairmax} and the expression in Theorem
\ref{thm:b_dyadic} depend only on $\psi:=|\alpha-\beta|/\sqrt{s}$.
\begin{defn}
We\emph{ say that a function $h:[0,\infty)\to[0,1]$ is an attainable
failure probability bound} if there exists a coupling $\{\{\tilde{X}_{\alpha,t}\}_{t\ge0}\}_{\alpha\in\mathbb{R}}$
of $\{\mathrm{BM}(\alpha)\}_{\alpha\in\mathbb{R}}$ such that
\[
\mathbf{P}\left(\exists\,t\ge s\,\mathrm{s.t.}\,\tilde{X}_{\alpha,t}\neq\tilde{X}_{\beta,t}\right)\le h\left(\frac{|\alpha-\beta|}{\sqrt{s}}\right)
\]
for all $\alpha,\beta,\in\mathbb{R}$, $s>0$. We say that $c>0$
is an \emph{attainable multiplicative gap} if $x\mapsto\mathrm{erf}(x\sqrt{c}/(2\sqrt{2}))$
is an attainable failure probability bound.
\end{defn}

One attainable failure probability bound is given in Theorem \ref{thm:b_dyadic}.
Corollary \ref{cor:b_tail}.\ref{enu:b_tail_unif} implies that $2e^{2}$
is an attainable multiplicative gap. A pairwise maximal coupling exists
if and only if $x\mapsto\mathrm{erf}(x/(2\sqrt{2}))$ is an attainable
failure probability bound, or equivalently, $1$ is an attainable
multiplicative gap.

We now prove a lower bound on any attainable failure probability bound
which implies a lower bound on the attainable multiplicative gap.
This implies the nonexistence of a pairwise maximal coupling.
\begin{thm}
\label{thm:fpb_bd}If $h$ is a failure probability bound, then for
any $0<s<t$, we have
\begin{align*}
 & \tilde{h}\left(\frac{2}{\sqrt{t}}\right)+\tilde{h}\left(\frac{2}{\sqrt{s}}\right)+2\tilde{h}\left(\frac{1}{\sqrt{s}}\right)\\
 & \ge\mathrm{erfc}\left(\frac{1}{\sqrt{2t}}\right)-\mathrm{erfc}\left(\frac{1}{\sqrt{2s}}\right)-\mathrm{erfc}\left(\frac{1}{4\sqrt{2(t-s)}}\right),
\end{align*}
where $\tilde{h}(x):=h(x)-\mathrm{erf}(x/(2\sqrt{2}))$, and $\mathrm{erfc}(x):=1-\mathrm{erf}(x)$.
Hence $x\mapsto\mathrm{erf}(x/(2\sqrt{2}))$ is not an attainable
failure probability bound. Moreover, $1.0025$ is not an attainable
multiplicative gap.
\end{thm}

\begin{proof}
For any $\alpha<\beta$, $s>0$, we have
\begin{align}
 & \mathbf{P}\left(X_{\beta,s}\le(\alpha+\beta)/2\;\mathrm{and}\;X_{\beta,s}\neq X_{\alpha,s}\right)+\mathbf{P}\left(X_{\alpha,s}\ge(\alpha+\beta)/2\;\mathrm{and}\;X_{\beta,s}\neq X_{\alpha,s}\right)\nonumber \\
 & =\mathbf{P}(X_{\beta,s}\le(\alpha+\beta)/2)+\mathbf{P}(X_{\alpha,s}\ge(\alpha+\beta)/2)\nonumber \\
 & \;\;\;-\left(\mathbf{P}\left(X_{\beta,s}\le(\alpha+\beta)/2\;\mathrm{and}\;X_{\beta,s}=X_{\alpha,s}\right)+\mathbf{P}\left(X_{\alpha,s}\ge(\alpha+\beta)/2\;\mathrm{and}\;X_{\beta,s}=X_{\alpha,s}\right)\right)\nonumber \\
 & \le\mathbf{P}(X_{\beta,s}\le(\alpha+\beta)/2)+\mathbf{P}(X_{\alpha,s}\ge(\alpha+\beta)/2)-\mathbf{P}(X_{\beta,s}=X_{\alpha,s})\nonumber \\
 & \stackrel{(a)}{\le}1-\mathrm{erf}\left(\frac{|\alpha-\beta|}{2\sqrt{2s}}\right)-\left(1-h\left(\frac{|\alpha-\beta|}{\sqrt{s}}\right)\right)\nonumber \\
 & =\tilde{h}\left(\frac{|\alpha-\beta|}{\sqrt{s}}\right),\label{eq:tail_agree}
\end{align}
where (a) is because $X_{\alpha,s}\sim\mathrm{N}(\alpha,s)$, $X_{\beta,s}\sim\mathrm{N}(\beta,s)$
and by the definition of the failure probability bound.

Let $0<s<t$. We have 
\begin{align*}
 & \mathbf{P}\left(X_{1,s}-X_{-1,s}\le1/2\;\mathrm{and}\;X_{1,s}\neq X_{-1,s}\right)\\
 & \le\mathbf{P}\left(X_{1,s}\le0\;\mathrm{and}\;X_{1,s}\neq X_{-1,s}\right)+\mathbf{P}\left(X_{-1,s}\ge0\;\mathrm{and}\;X_{1,s}\neq X_{-1,s}\right)\\
 & \;\;\;+\mathbf{P}\left(X_{1,s}\le1/2\;\mathrm{and}\;X_{-1,s}\ge1/2\;\mathrm{and}\;X_{1,s}\neq X_{-1,s}\right)\\
 & \stackrel{(a)}{\le}\mathbf{P}\left(X_{1,s}\le0\;\mathrm{and}\;X_{1,s}\neq X_{-1,s}\right)+\mathbf{P}\left(X_{-1,s}\ge0\;\mathrm{and}\;X_{1,s}\neq X_{-1,s}\right)\\
 & \;\;\;+\mathbf{P}\left(X_{1,s}\le1/2\;\mathrm{and}\;X_{1,s}\neq X_{0,s}\right)+\mathbf{P}\left(X_{-1,s}\ge1/2\;\mathrm{and}\;X_{-1,s}\neq X_{0,s}\right)\\
 & \stackrel{(b)}{\le}\tilde{h}(2/\sqrt{s})+2\tilde{h}(1/\sqrt{s}),
\end{align*}
where (a) is because if $X_{1,s}\neq X_{-1,s}$, then either $X_{1,s}\neq X_{0,s}$
or $X_{-1,s}\neq X_{0,s}$, and (b) is by applying \eqref{eq:tail_agree}
on $(\alpha,\beta,s)\leftarrow(-1,1,s)$, $(0,1,s)$ and $(-1,0,s)$
respectively. Hence,
\begin{align*}
 & \mathbf{P}\left(X_{1,t}=X_{-1,t}\;\mathrm{and}\;X_{1,s}\neq X_{-1,s}\right)\\
 & \le\mathbf{P}\left(X_{1,s}-X_{-1,s}\le1/2\;\mathrm{and}\;X_{1,s}\neq X_{-1,s}\right)\\
 & \;\;+\mathbf{P}(X_{1,t}-X_{1,s}\le-1/4)+\mathbf{P}(X_{-1,t}-X_{-1,s}\ge1/4)\\
 & \le\tilde{h}(2/\sqrt{s})+2\tilde{h}(1/\sqrt{s})+\mathrm{erfc}\left(\frac{1/4}{\sqrt{2(t-s)}}\right).
\end{align*}
Therefore,
\begin{align*}
 & 1-h(2/\sqrt{t})\\
 & \le\mathbf{P}(X_{1,t}=X_{-1,t})\\
 & \le\mathbf{P}(X_{1,s}=X_{-1,s})+\mathbf{P}\left(X_{1,t}=X_{-1,t}\;\mathrm{and}\;X_{1,s}\neq X_{-1,s}\right)\\
 & \le\mathrm{erfc}\left(\frac{2}{2\sqrt{2s}}\right)+\tilde{h}(2/\sqrt{s})+2\tilde{h}(1/\sqrt{s})+\mathrm{erfc}\left(\frac{1/4}{\sqrt{2(t-s)}}\right)\\
 & =\mathrm{erfc}\left(\frac{1}{\sqrt{2s}}\right)+\mathrm{erfc}\left(\frac{1}{4\sqrt{2(t-s)}}\right)+\tilde{h}(2/\sqrt{s})+2\tilde{h}(1/\sqrt{s}).
\end{align*}
Hence,
\begin{align*}
 & \tilde{h}\left(\frac{2}{\sqrt{t}}\right)+\tilde{h}\left(\frac{2}{\sqrt{s}}\right)+2\tilde{h}\left(\frac{1}{\sqrt{s}}\right)\\
 & \ge\mathrm{erfc}\left(\frac{1}{\sqrt{2t}}\right)-\mathrm{erfc}\left(\frac{1}{\sqrt{2s}}\right)-\mathrm{erfc}\left(\frac{1}{4\sqrt{2(t-s)}}\right).
\end{align*}
Note that the above lower bound can be positive (e.g. it is at least
$0.0019$ when $s=0.33$, $t=0.3348$), and thus $x\mapsto\mathrm{erf}(x/(2\sqrt{2}))$
is not an attainable failure probability bound.

It can be verified numerically that $x\mapsto\mathrm{erf}(x\sqrt{c}/(2\sqrt{2}))$
does not satisfy the above inequality when $c=1.0025$, $s=0.2361$,
$t=0.2408$. Hence $1.0025$ is not an attainable multiplicative gap.
\end{proof}
\medskip{}

We conjecture that the dyadic grand coupling is optimal in the following
sense.
\begin{conjecture}
\label{conj:optimal}If $h$ is a failure probability bound, then
for any $\psi>0$, we have
\[
h(\psi)\ge\int_{\psi/2}^{\infty}\left(\sum_{k=1}^{\infty}2(-1)^{k+1}\exp\left(-\frac{k^{2}\pi^{2}}{2\zeta^{2}}\right)\right)\frac{\zeta^{-2}\psi}{\ln2}\left(\min\left\{ \zeta\psi^{-1},\,1\right\} -\frac{1}{2}\right)\mathrm{d}\zeta,
\]
i.e., the attainable failure probability bound given in Theorem \ref{thm:b_dyadic}
is pointwise optimal.
\end{conjecture}

\smallskip{}

Loosely speaking, the dyadic grand coupling is ``locally a reflection
coupling'', in the sense that the coupled processes after the time
of each coalescence point can be obtained by performing the reflection
coupling between adjacent pairs of coalescence points (see Figure
\ref{fig:sample}). In Conjecture \ref{conj:optimal}, we raise the
question whether such ``locally optimal'' coupling is globally optimal.

It may also be of interest to find the smallest attainable multiplicative
gap. Theorem \ref{thm:fpb_bd} and the numerical evidence in Figure
\ref{fig:bdr} show that the infimum of the set of attainable multiplicative
gaps is between $1.0025$ and $1.5$.\smallskip{}

\section{Conclusions and Discussion}

We constructed a coupling of $\{\mathrm{BM}(\alpha)\}_{\alpha\in\mathbb{R}}$,
such that the coupling for any pair of the coupled processes is close
to being maximal. While it is shown that a pairwise exactly maximal
coupling does not exist, we conjecture that our coupling is optimal
among couplings of $\{\mathrm{BM}(\alpha)\}_{\alpha\in\mathbb{R}}$
in the sense of attainable failure probability bounds.

One future direction is to generalize the construction to Brownian
motions in $\mathbb{R}^{n}$. While we can couple each coordinate
independently using the dyadic grand coupling, this may not be the
optimal construction.

Another direction is to consider Brownian motions with initial distributions
(rather than fixed starting points), i.e., the collection of processes
is $\{\mathrm{BM}(P)\}_{P\in\mathcal{P}(\mathbb{R})}$, where $\mathcal{P}(\mathbb{R})$
is the set of distributions over $\mathbb{R}$, and $\mathrm{BM}(P)$
is the Brownian motion with initial distribution $P$. One simple
construction is to first couple the starting point by the quantile
coupling, then apply the dyadic grand coupling, i.e., $X_{P,0}:=F_{P}^{-1}(U)$,
where $U\sim\mathrm{Unif}[0,1]$, and $X_{P,t}:=X_{X_{P,0},t}$ for
$t>0$, where $X_{X_{P,0},t}$ is given by the dyadic grand coupling
with starting point $X_{P,0}$. Another construction is to use the
sequential Poisson functional representation \cite{li2019pairwise}
instead of the quantile coupling, since it is more suitable for minimizing
concave costs (it is shown in Appendix \ref{sec:pf_concave} that
the probability of failure in Theorem \ref{thm:b_dyadic} is concave
in $|\alpha-\beta|$).

\smallskip{}

\section{Acknowledgements}

The authors acknowledge support from the NSF grants CNS-1527846, CCF-1618145,
CCF-1901004, the NSF Science \& Technology Center grant CCF-0939370
(Science of Information), and the William and Flora Hewlett Foundation
supported Center for Long Term Cybersecurity at Berkeley.

\medskip{}

\appendix
\[
\]

\section{Proof of the Claim in Definition \ref{def:dyadic}\label{sec:pf_dyadic}}

We first prove that for any $\theta\in[0,1]$, $\alpha\in\mathbb{R}$,
we have $G_{\theta,\alpha,j}=W_{j}$ for all sufficiently large $j$,
as long as $\sup\{j:W_{j}=1\}=\sup\{j:W_{j}=-1\}=\infty$. Let $k_{1}\in\mathbb{Z}$
be such that $2^{k_{1}}>4|\alpha|$ and $W_{k_{1}}=1$, and $k_{-1}\in\mathbb{Z}$
be such that $2^{k_{-1}}>4|\alpha|$ and $W_{k_{-1}}=-1$. Assume
$j\ge\max\{k_{1},k_{-1}\}$. We have
\begin{align*}
 & \alpha-\sum_{k=-\infty}^{j-1}W_{k}2^{k+\theta-1}\\
 & \le\alpha+\sum_{k\le j-1,\,k\neq k_{1}}2^{k+\theta-1}-2^{k_{1}+\theta-1}\\
 & =\alpha+2^{j+\theta-1}-2^{k_{1}+\theta}\\
 & <2^{j+\theta-1},
\end{align*}
where the last inequality is by the definition of $k_{1}$. Similarly,
$\alpha-\sum_{k=-\infty}^{j-1}W_{k}2^{k+\theta-1}>-2^{j+\theta-1}$.
Hence,
\[
\left(\alpha-\sum_{k=-\infty}^{j-1}W_{k}2^{k+\theta-1}\!+2^{j+\theta-1}\right)\mathrm{mod}\;2^{j+\theta+1}\in[0,2^{j+\theta}),
\]
and $G_{\theta,\alpha,j}=W_{j}$.

We then prove \eqref{eq:dyadic_partial}. We will prove by induction
that for all $j\in\mathbb{Z}$,
\begin{align*}
 & \left\lfloor 2^{-(j+\theta)}\left(\alpha-\sum_{k=-\infty}^{j-1}W_{k}2^{k+\theta-1}\!+2^{j+\theta-1}\right)\right\rfloor \\
 & =2^{-(j+\theta)}\sum_{k=j}^{\infty}(W_{k}-G_{\theta,\alpha,k})2^{k+\theta-1}.
\end{align*}
If $j\ge\max\{k_{1},k_{-1}\}$, then $-2^{j+\theta-1}<\alpha-\sum_{k=-\infty}^{j-1}W_{k}2^{k+\theta-1}<2^{j+\theta-1}$,
and also $W_{k}=G_{\theta,\alpha,k}$ for $k\ge j$, and thus both
sides in the induction hypothesis are $0$.

Assume the induction hypothesis is true for $j+1$. If $W_{j}=G_{\theta,\alpha,j}$,
then
\[
\left(\alpha-\sum_{k=-\infty}^{j-1}W_{k}2^{k+\theta-1}\!+2^{j+\theta-1}\right)\mathrm{mod}\;2^{j+\theta+1}\in[0,2^{j+\theta}),
\]
and hence
\begin{align*}
 & \left\lfloor 2^{-(j+\theta)}\left(\alpha-\sum_{k=-\infty}^{j-1}W_{k}2^{k+\theta-1}\!+2^{j+\theta-1}\right)\right\rfloor \\
 & \stackrel{(a)}{=}2\left\lfloor 2^{-(j+\theta+1)}\left(\alpha-\sum_{k=-\infty}^{j-1}W_{k}2^{k+\theta-1}\!+2^{j+\theta-1}+(1-W_{j})2^{j+\theta-1}\right)\right\rfloor \\
 & =2\left\lfloor 2^{-(j+\theta+1)}\left(\alpha-\sum_{k=-\infty}^{j}W_{k}2^{k+\theta-1}\!+2^{j+\theta}\right)\right\rfloor \\
 & \stackrel{(b)}{=}2\cdot2^{-(j+1+\theta)}\sum_{k=j+1}^{\infty}(W_{k}-G_{\theta,\alpha,k})2^{k+\theta-1}\\
 & =2^{-(j+\theta)}\sum_{k=j}^{\infty}(W_{k}-G_{\theta,\alpha,k})2^{k+\theta-1},
\end{align*}
where (a) is because $(1-W_{j})2^{j+\theta-1}\in[0,2^{j+\theta}]$,
and (b) is by the induction hypothesis for $j+1$. Therefore the induction
hypothesis holds for $j$. If $W_{j}=-G_{\theta,\alpha,j}$, then
\[
\left(\alpha-\sum_{k=-\infty}^{j-1}W_{k}2^{k+\theta-1}\!+2^{j+\theta-1}\right)\mathrm{mod}\;2^{j+\theta+1}\in[2^{j+\theta},2^{j+\theta+1}),
\]
and hence
\begin{align*}
 & \left\lfloor 2^{-(j+\theta)}\left(\alpha-\sum_{k=-\infty}^{j-1}W_{k}2^{k+\theta-1}\!+2^{j+\theta-1}\right)\right\rfloor \\
 & \stackrel{(a)}{=}2\left\lfloor 2^{-(j+\theta+1)}\left(\alpha-\sum_{k=-\infty}^{j-1}W_{k}2^{k+\theta-1}\!+2^{j+\theta-1}+(1-W_{j})2^{j+\theta-1}\right)\right\rfloor +W_{j}\\
 & =2\left\lfloor 2^{-(j+\theta+1)}\left(\alpha-\sum_{k=-\infty}^{j}W_{k}2^{k+\theta-1}\!+2^{j+\theta}\right)\right\rfloor +W_{j}\\
 & =2\cdot2^{-(j+1+\theta)}\sum_{k=j+1}^{\infty}(W_{k}-G_{\theta,\alpha,k})2^{k+\theta-1}+W_{j}\\
 & =2^{-(j+\theta)}\sum_{k=j}^{\infty}(W_{k}-G_{\theta,\alpha,k})2^{k+\theta-1},
\end{align*}
where (a) can be deduced by considering whether $W_{j}=1$ or $-1$.
Therefore the induction hypothesis holds for $j$.

Hence the induction hypothesis holds for all $j\in\mathbb{Z}$, and
\[
\sum_{k=j}^{\infty}(W_{k}-G_{\theta,\alpha,k})2^{k+\theta-1}\!-2^{j+\theta-1}\le\alpha-\!\sum_{k=-\infty}^{j-1}W_{k}2^{k+\theta-1}\!<\sum_{k=j}^{\infty}(W_{k}-G_{\theta,\alpha,k})2^{k+\theta-1}\!+2^{j+\theta-1}.
\]
Letting $j\to-\infty$, we have $\sum_{j=-\infty}^{\infty}(W_{j}-G_{\theta,\alpha,j})2^{j+\theta-1}=\alpha$.

\[
\]

\section{Proof that the Expression in Theorem \ref{thm:b_dyadic} is concave
in $|\alpha-\beta|$\label{sec:pf_concave}}

Let $h(\psi):=\mathbf{P}(\exists\,\gamma_{1},\gamma_{2}\in[0,\psi],t\ge1\,\mathrm{s.t.}\,X_{\gamma_{1},t}\neq X_{\gamma_{2},t})$
be given in Theorem \ref{thm:b_dyadic}. Let $0<\psi_{1}<\psi_{2}$.
Fix $l>\psi_{2}$. Let $(A_{1},B_{1}),(A_{2},B_{2}),\ldots,(A_{N},B_{N})\subseteq[0,l]$
be maximal open intervals with length at least $\psi_{1}$, sorted
in ascending order, such that $\{X_{\gamma,t}\}_{t\ge1}$ is constant
within each interval (i.e., for any $i=1,\ldots,N$, we have $B_{i}-A_{i}\ge\psi_{1}$,
$\{X_{\gamma_{1},t}\}_{t\ge1}=\{X_{\gamma_{2},t}\}_{t\ge1}$ for any
$\gamma_{1},\gamma_{2}\in(A_{i},B_{i})$, and any open interval in
$[0,l]$ that is a proper superset of $(A_{i},B_{i})$ does not have
this property). For any $\rho\in[0,1]$, let $\psi:=\rho\psi_{1}+(1-\rho)\psi_{2}$.
We have
\begin{align*}
h(\psi) & =\frac{1}{l-\psi}\int_{0}^{l-\psi}\mathbf{P}(\exists\,\gamma_{1},\gamma_{2}\in[x,x+\psi],t\ge1\,\mathrm{s.t.}\,X_{\gamma_{1},t}\neq X_{\gamma_{2},t})\mathrm{d}x\\
 & =1-\frac{1}{l-\psi}\mathbf{E}\left[\sum_{i=1}^{N}\max\{B_{i}-A_{i}-\psi,0\}\right]\\
 & \stackrel{(a)}{\ge}1-\frac{1}{l-\psi}\mathbf{E}\left[\sum_{i=1}^{N}\left(\rho\max\{B_{i}-A_{i}-\psi_{1},0\}+(1-\rho)\max\{B_{i}-A_{i}-\psi_{2},0\}\right)\right]\\
 & =1-\frac{1}{l-\psi}\left(\rho\mathbf{E}\left[\sum_{i=1}^{N}\max\{B_{i}-A_{i}-\psi_{1},0\}\right]+(1-\rho)\mathbf{E}\left[\sum_{i=1}^{N}\max\{B_{i}-A_{i}-\psi_{2},0\}\right]\right)\\
 & =1-\frac{1}{l-\psi}\left(\rho(l-\psi_{1})(1-h(\psi_{1}))+(1-\rho)(l-\psi_{2})(1-h(\psi_{2}))\right),
\end{align*}
where (a) is by the convexity of $x\mapsto\max\{\gamma-x,0\}$. Letting
$l\to\infty$, we have $h(\psi)\ge\rho h(\psi_{1})+(1-\rho)h(\psi_{2})$.
Hence $h$ is concave on $(0,\infty)$. Since $h$ is non-decreasing,
$h$ is concave on $[0,\infty)$.

\medskip{}

\[
\]

\bibliographystyle{IEEEtran}
\bibliography{ref}

\end{document}